\documentclass[article]{siamart250211}
\newsiamremark{assumption}{Assumption}
\newsiamremark{remark}{Remark}
\newsiamremark{example}{Example}

\usepackage{amssymb,amsmath}
\usepackage{empheq}
\usepackage{stackrel}
\usepackage[ruled,algo2e]{algorithm2e}
\usepackage{bm}
\usepackage[ulem=normalem,draft,authormarkup=none,deletedmarkup=sout]{changes}

\ifpdf
\hypersetup{
  pdftitle={Second-Order Conditions for Semilinear Infinite-Horizon Control-Constrained Parabolic Problems without Regularization},
  pdfauthor={E.~Casas, and N.~Jork}
}
\fi

\newcommand{\dx}{\,\mathrm{dx}}
\newcommand{\dt}{\,\mathrm{dt}}

\newcommand{\Pb}{\mbox{\rm (P)}\xspace}
\newcommand{\PbT}{\mbox{\rm (P$_T$)}\xspace}
\newcommand{\Pbr}{\mbox{\rm (P$_\rho$)}\xspace}
\newcommand{\PbTr}{\mbox{\rm (P$_{T,\rho}$)}\xspace}

\newcommand{\uad}{U_{\rm ad}}
\newcommand{\uadt}{U_{\rm ad, T}}
\newcommand{\Kad}{K_{\rm ad}}

\title{Second-Order Conditions for Infinite-Horizon Semilinear Parabolic Control Problems without Tikhonov Regularization\thanks{Submitted to the editors DATE} \funding{The first author was supported by MICIU/AEI/10.13039/501100011033/ under research project PID2023-147610NB-I00.}}

\author{Eduardo Casas \thanks{Departamento de Matem\'{a}tica
Aplicada y Ciencias de la Computaci\'{o}n,  Universidad de Cantabria, Spain. (\email{eduardo.casas@unican.es})}
\and Nicolai Jork \thanks{Department of Mathematics, Eberhard-Karls-Universit\"at T\"ubingen, Germany.
(\email{nicolai.jork@uni-tuebingen.de})}
}

\headers{Second-Order Optimality Conditions}{E. Casas \and N. Jork}

\begin{document}

\maketitle
\begin{abstract}
We consider semilinear parabolic optimal control problems subject to Neumann boundary conditions, control constraints, and an infinite time horizon. The control constraints are pointwise in time, but they can be pointwise or integral in the space variable. Crucially, the optimal control problem does not include a Tikhonov regularization in the cost functional, which provides a major difficulty in the extension of the classical finite-horizon theory to infinite-horizon optimal control problems. As a consequence of our findings, we establish a sufficient second-order optimality condition and prove that local optimal states of the finite-horizon problems approximate local optimal states to the infinite-horizon problem as the horizon tends to infinity.
\end{abstract}

\begin{keywords}
semilinear parabolic equation, infinite time horizon, optimal control, control constraints, second order optimality  conditions.
\end{keywords}

\begin{MSCcodes}
35K58, 
35Q93, 
49K20.  
\end{MSCcodes}

\section{Introduction}\label{S1}
\setcounter{equation}{0}

This paper analyzes an infinite-horizon optimal control problem governed by a semilinear parabolic equation. The cost functional does not include the regularizing Tikhonov term. Our main goal is to derive second-order sufficient optimality conditions and apply them to obtain error estimates for approximating the problem with finite-horizon control problems. Second-order optimality conditions and stability results for semilinear parabolic optimal control problems with control constraints are well established in the finite-horizon setting; see \cite{CT2015, Casas-Troltzsch2016, CT2022,  DJV, RT2003} and the monograph \cite{T2010}. In contrast, the infinite-horizon case remains far less developed. With a focus on necessary optimality conditions and in the presence of ODE-constraints, these were considered by Halkin in \cite{zbMATH03471770}. Infinite-horizon problems with ODE constraints and control constraints, incorporating a discount factor, were investigated in \cite{zbMATH02205762, zbMATH06165819}. For related problems that contain state constraints, we refer to \cite{zbMATH07863175, zbMATH07076355}.
While the above works focus on ODE constraints, early results for infinite horizon problems subject to PDE-constraints were established in \cite[Chapter 9]{zbMATH00051863}, \cite[Chapter III.6]{zbMATH03322278}, \cite{zbMATH07086583} and \cite{zbMATH07700209}. More recent contributions include \cite{Bociu-Casas2026, zbMATH07563683, Casas-Kunisch2023C}. 
All of the above mentioned works contain a Tikhonov regularization term in the cost functional. The appearance of this term helps to address the additional difficulties introduced by the infinite horizon and makes results analogous to the finite-horizon theory available. However, this is not the case in \cite{Bociu-Casas2026}, due to the presence of pointwise state constraints.

Without regularization, however, the problem admits further challenges: standard arguments relying on coercivity with respect to the controls fail.
This paper addresses this gap in the literature by studying an unregularized infinite-horizon problem where the cost functional is of $L^2$-state-tracking type. The state equation is a semilinear parabolic PDE with a Neumann boundary condition on a bounded spatial domain $\Omega \subset \mathbb{R}^n$.
In the absence of a Tikhonov term, control constraints arise naturally in order to bound the size of the control. Typically, box constraints are imposed. Besides box constraints, constraints that are pointwise in time but integral in space were studied for optimal control problems that track the velocity of Navier–Stokes flow in a bounded two-dimensional domain; see \cite{Casas-Chrysafinos2026, zbMATH07362115, zbMATH01367200}.

Our main contributions are as follows. First, we derive a second-order sufficient optimality condition for strong local optimality, for both classes of admissible controls considered in this paper. To this end, we extend finite-horizon techniques from \cite{Casas-Chrysafinos2026} using infinite-horizon results from \cite{Casas-Kunisch2023C}. This has required the implementation of new ideas. Second, we show that local solutions of the finite-horizon problems approximate local solutions of the infinite-horizon problem as the time horizon tends to infinity. As an application of the established second-order sufficient condition, we obtain error estimates for the difference between the the finite- and infinite-horizon optimal states.

The paper is organized as follows. In Section \ref{S2} we introduce notation and assumptions, recall well-posedness and differentiability properties of the control-to-state mapping, derive the estimates needed for the linearized and adjoint equations, and state the first-order necessary optimality conditions. Section \ref{S3} is devoted to the derivation of a second-order sufficient optimality condition. In Section \ref{S4} we investigate the approximation of the infinite-horizon problem by finite-horizon problems, extending the corresponding results for Tikhonov-regularized problems from \cite{Casas-Kunisch2023C}. Appendix A collects auxiliary lemmas.

We come to a precise definition of the optimal control problem studied in this paper
\[
\Pb\quad \min_{u \in \uad^j} J(u):=\frac{1}{2}\int_0^\infty\int_\Omega(y_u - y_d)^2\dx\dt,
\]
where $y_u$ is the solution of the following parabolic state equation with homogeneous Neumann boundary condition
\begin{equation}
\left\{\begin{array}{l}\displaystyle \frac{\partial y}{\partial t} + Ay + f(y) = g + u\chi_\omega\ \text{ in } Q = \Omega \times (0,\infty),\\
\partial_{n_A}y = 0 \ \text{ on } \Sigma = \Gamma \times (0,\infty),\ \ y(0) = y_0 \ \text{ in } \Omega.\end{array}\right.
\label{E1.1}
\end{equation}
Here $\omega$ is a measurable subset of $\Omega$ with $|\omega| > 0$.
Motivated by the works \cite{zbMATH07563683, zbMATH07490298, Casas-Kunisch2023C}, we choose the set of admissible controls to be  $$\uad^j = \{u\in L^2(Q_\omega)\vert \ u(\cdot, t) \in K_{\mathrm{ad}}^j(t)\text{ a.a. } t\in (0,\infty)\}$$ where $\Kad^j(t)$ is a closed, convex and bounded set in $L^2(\omega)$. We consider the following two possibilities for the set $\Kad^j$. For the first choice, we fix a measurable function $\gamma:[0,\infty)\longrightarrow (0,\infty)$ while for the second choice, the box-constraint, we fix two measurable functions $\alpha, \beta:\omega \times [0,\infty) \longrightarrow \mathbb{R}$, such that the set $K_{ad}^j$, $j\in\{1,2\}$, is given by
\begin{align}
K_{\mathrm{ad}}^1(t) &:= B_{\gamma(t)}:= \{ v\in L^2(\omega) \mid \|v\|_{L^2(\omega)}\le \gamma(t) \}
  \label{E1.2},\\
K_{\mathrm{ad}}^2(t)
&:=\{v\in L^2(\omega) \ \vert\ \alpha(\cdot,t)\le v \le \beta(\cdot,t)\ \text{a.e. in }\omega\}
  \label{E1.3}.
\end{align}
Assumptions on the involved function will be given in the next section.

\section{Assumptions and preliminary results}\label{S2}
\setcounter{equation}{0}

We make the following assumptions on the control problem \Pb.
\begin{assumption}\label{A2.1}
Hereafter $\Omega$ denotes a bounded open subset of $\mathbb{R}^n$, $1 \le n \le 3$, with a Lipschitz boundary $\Gamma$. In the case $n = 1$, $\Omega$ is a real interval. $A$ stands for the partial differential operator defined in $\Omega$, which is assumed to be symmetric and is given by
\[
Ay = -\sum_{i, j = 1}^n\partial_{x_j}[a_{ij}(x)\partial_{x_i}y] + a_0y
\]
with $a_0, a_{ij} \in L^\infty(\Omega)$ satisfying
\begin{equation}
\left\{\begin{array}{l}\exists \Lambda > 0 \text{ such that } \sum_{i, j = 1}^na_{ij}(x)\xi_i\xi_j \ge \Lambda|\xi|^2\ \ \forall \xi \in \mathbb{R}^n,\text{  for a.a. } x \in \Omega,\\ \ a_0(x) \ge 0 \text{ and } a_0 \not\equiv 0.\end{array}\right.
\label{E2.1}
\end{equation}
\end{assumption}

We observe that \eqref{E2.1} implies the existence of $\Lambda_A > 0$ such that
\begin{equation}
\Lambda_A\|y\|^2_{H^1(\Omega)} \le \int_\Omega\Big[\sum_{i, j = 1}^na_{ij}(x)\partial_{x_i}y(x)\partial_{x_j}y(x) + a_0(x)y^2(x)\Big]\dx \quad \forall y \in H^1(\Omega).
\label{E2.2}
\end{equation}

\begin{assumption}
The nonlinear function $f:\mathbb{R} \to\mathbb{R}$ is of class $C^2$, $f'(s) \ge 0$ for all $s \in \mathbb{R}$, and $f(0) = 0$.  We also assume that $g \in L^2(Q) \cap  L^p(0,\infty;L^2(\Omega))$, where $p$ is fixed along this paper as follows
\[
p \in (\frac{4}{4 - n},\infty) \text{ if $n = 2$ or 3, and } p \in [2,\infty) \text{ if } n = 1.
\]
For the initial condition in \eqref{E1.1}, we assume $y_0\in L^\infty(\Omega)$.
\label{A2}
\end{assumption}
 In the following, we define $Q_\omega: = \omega \times (0,\infty)$.
\begin{assumption}
The positive function $\gamma$ appearing in the constraint \eqref{E1.2} satisfies $\gamma\in L^1(0,\infty)\cap L^\infty(0,\infty)$. The functions appearing in the constraint \eqref{E1.3}, satisfy $\alpha, \beta \in L^2(Q_\omega) \cap L^\infty(Q_\omega)$. We also suppose that $\alpha(x,t) \le 0 \le \beta(x,t)$ and $\alpha(x,t) < \beta(x,t)$ a.e. in $Q_\omega$. Additionally we assume that $y_d \in L^2(Q) \cap  L^p(0,\infty;L^2(\Omega))$.
\label{A3}
\end{assumption}

From this assumption we infer that
\begin{align*}
&\uad^1 \subset L^1(0,\infty;L^2(\omega)) \cap L^\infty(0,\infty;L^2(\omega)) \subset L^p(0,\infty;L^2(\omega)),\\
&\uad^2 \subset L^2(Q_\omega) \cap L^\infty(0,\infty;L^2(\omega)) \subset L^p(0,\infty;L^2(\omega)).
\end{align*}

Following the standard notation, for every $T \in (0,\infty]$ we consider the Hilbert space $W(0,T) = \{y \in L^2(0,T;H^1(\Omega)) : \partial_ty \in L^2(0,T;H^1(\Omega)^*)\}$ endowed with the norm
\[
\|y\|_{W(0,T)} = \Big(\|y\|^2_{L^2(0,T;H^1(\Omega))} + \|\partial_ty\|^2_{L^2(0,T;H^1(\Omega)^*)}\Big)^{\frac{1}{2}}.
\]
Moreover, the continuous embeddings $W(0,T) \subset C([0,T];L^2(\Omega))$ and $W(0,T) \subset L^2(Q_T)$ are well known, where $Q_T := \Omega \times (0,T)$. Moreover $W(0,T) \subset L^2(Q_T)$ is compact if $T < \infty$. Now, we analyze the state equation \eqref{E1.1}.

\begin{definition}
We say that $y$ is a solution of \eqref{E1.1} if $y \in W(0,\infty) \cap L^\infty(Q)$ and the following identity holds
\begin{equation}
\langle\frac{\partial y}{\partial t},z\rangle + \int_\Omega\Big[\sum_{i, j = 1}^na_{ij}(x)\partial_{x_i}y\partial_{x_j}z(x) + a_0(x)yz + f(y)z\Big]\dx = \int_\Omega[g + u\chi_\omega]z\dx
\label{E2.3}
\end{equation}
\label{D2.1}
\end{definition}
for all $z \in H^1(\Omega)$ and almost everywhere in time, where $\langle\cdot, \cdot\rangle$ denotes the duality pair between $H^1(\Omega)^*$ and $H^1(\Omega)$.

In the next theorem we establish the existence, uniqueness, and regularity of the solutions $y_u$ to the state equation \eqref{E1.1}. 

\begin{theorem}\label{T2.1}
For every $u \in L^2(Q_\omega) \cap L^p(0,\infty;L^2(\omega))$, \eqref{E1.1} has a unique solution $y_u \in W(0,\infty) \cap L^\infty(Q)$. Moreover there exist constants $M_1$ and $M_2$ independent of $u$, $g$, and $y_0$ such that
\begin{align}
&\|y_u\|_{L^2(0,\infty;H^1(\Omega))} + \|y_u\|_{L^\infty(0,\infty;L^2(\Omega))} \le M_1\Big(\|y_0\|_{L^2(\Omega)} + \|g + u\chi_\omega\|_{L^2(Q)}\Big),\label{E2.4}\\
&\|y_u\|_{L^\infty(Q)} \le M_2\Big(\|y_0\|_{L^\infty(\Omega)} + \|g + u\chi_\omega\|_{L^2(Q)} + \|g + u\chi_\omega\|_{L^p(0,\infty;L^2(\omega))}\Big). \label{E2.5}
\end{align}
In addition, there exists a constant $M_{ad}$ such that
\begin{equation}
\|y_u\|_{W(0,\infty)} + \|y_u\|_{L^\infty(Q)} \le M_{ad}\quad \forall u \in \uad^j,\ j\in\{1,2\}.
\label{E2.6}
\end{equation}
Finally, if $\{u_k\}_{k = 1}^\infty \subset \uad^j$, converges weakly to $u$ in $L^2(Q_\omega)$ then $y_{u_k} \rightharpoonup y_u$ in $W(0,\infty)$ and $\lim_{k \to \infty}\|y_{u_k} - y_u\|_{L^\infty(Q_\omega)} = 0$.
\end{theorem}

\begin{proof}
The reader is referred to \cite{CKT2025} for the proof of this theorem. We only need to prove the last statement. It is clear that both constraints \eqref{E1.2} and \eqref{E1.3} yield the  regularity $u \in L^2(Q_\omega) \cap L^\infty(0,\infty;L^2(\omega)) \subset L^p(0,\infty;L^2(\omega))$ and From \eqref{E2.6} we infer the boundedness of $\{y_{u_k}\}_{k = 1}^\infty$ in $W(0,\infty) \cap L^\infty(Q)$. Then, we can take a subsequence, denoted in the same way, such that $y_{u_k} \stackrel{*}{\rightharpoonup} y$ in $W(0,\infty) \cap L^\infty(Q)$. This implies the strong convergence $y_{u_k} \to y$ in $L^2(Q_T)$ for every $T < \infty$. Then, it is straightforward to pass to the limit in the state equation satisfied by $y_{u_k}$ in $Q_T$ to deduce that $y{\mid_{Q_T}} = {y_u}_{\mid_{Q_T}}$ and, hence, $y = y_u$. Consequently, the whole sequence converges weakly to $y_u$ in $W(0,\infty)$. Let us prove the strong convergence in $L^\infty(Q)$. Let us denote $z_k = y_{u_k} - y_u$. Subtracting the equations satisfied by $y_{u_k}$ and $y_u$ we get
\[
\left\{\begin{array}{l}\displaystyle \frac{\partial z_k}{\partial t} + Az_k + f'(\hat y_k)z_k = (u_k - u)\chi_\omega \text{ in } Q,\\
\partial_{n_A}z_k = 0 \ \text{ on } \Sigma,\ \ z_k(0) = 0 \ \text{ in } \Omega,\end{array}\right.
\]
where $\hat y_k = y_u + \theta_k(y_{u_k} - y_u)$ for some measurable function $\theta_k:Q \longrightarrow [0,1]$. Then, for every $T < \infty$ we deduce from \cite{DER2017} or \cite[\S III.10]{Lad-Sol-Ura68} the existence of $\mu \in (0,1)$ and a constant $\hat C$ such that 
\[
\|z_k\|_{C^{0,\mu}(\bar Q_T)}\! \le\! C_1\|u_k - u\|_{L^p(0,\infty;L^2(\omega))}\! \le\! C_1\|u_k - u\|^{\frac{p - 2}{p}}_{L^\infty(0,\infty;L^2(\omega))}\|u_k - u\|^{\frac{2}{p}}_{L^2(Q_\omega)}\! \le\! C_2.
\]
Using the compactness of the embedding $C^{0,\mu}(\bar Q_T) \subset C(\bar Q_T)$ we infer the strong convergence $\lim_{k \to \infty}\|y_{u_k} - y_u\|_{C(\bar Q_T)} = 0$ for every $T < \infty$. To study what happens in $Q^T = \Omega \times (T,\infty)$, we observe that according to the assumptions on $\uad^j$, $j = 1, 2$, there exists a function $h \in L^2(0,\infty) \cap L^\infty(0,\infty)$ such that $\|v(t)\|_{L^2(\omega)} \le h(t)$ for almost all $t \in (0,\infty)$ and all $v \in \uad^j$. We can take $h = \gamma$ for $j = 1$ and $h = \max\{|\alpha|, \beta\}$ for $j = 2$. Then we have that for every $T < \infty$ such that
\begin{align*}
&\|z_k\|_{L^\infty(Q^T)} \le C_3\Big(\|z_k(T)\|_{L^\infty(\Omega)} + \|u_k - u\|_{L^2(\omega \times (T,\infty))} + \|u_k - u\|_{L^p(T,\infty;L^2(\omega))}\Big)\\
& \le C_3\Big(\|z_k(T)\|_{L^\infty(\Omega)} + \|h\|_{L^2(T;\infty)} + \|h\|_{L^p(T,\infty)}\Big) \quad \forall T < \infty.
\end{align*}
Given $\varepsilon > 0$ we select $T_\varepsilon < \infty$ such that
\[
C_3\Big(\|h\|_{L^2(T_\varepsilon;\infty)} + \|h\|_{L^p(T_\varepsilon,\infty)}\Big) < \frac{\varepsilon}{2}.
\]
Now, we select an integer $k_\varepsilon$ such that $\|z_k\|_{C(\bar Q_{T_\varepsilon})} < \frac{\varepsilon}{2\max\{1,C_3\}}$. Then, we have
\[
\|z_k\|_{L^\infty(Q^{T_\varepsilon})} \le C_3\Big(\|z_k(T_\varepsilon)\|_{L^\infty(\Omega)} + \|h\|_{L^2(T_\varepsilon;\infty)} + \|h\|_{L^p(T_\varepsilon,\infty)}\Big) < \varepsilon\quad \forall k \ge k_\varepsilon.
\]
This proves the desired convergence.
\end{proof}
We denote $G:L^2(Q_\omega) \cap L^p(0,\infty;L^2(\omega)) \longrightarrow W(0,\infty) \cap L^\infty(Q)$ the mapping associating to every control the corresponding state: $G(u) = y_u$. According to Theorem \ref{T2.1}, this mapping is well defined and from \cite{Casas-Kunisch2023C} we infer that it is of class $C^2$ and $z_{u,v} = G'(u)v$ and $z_{u,(v_1,v_2)} = G''(u)(v_1,v_2)$ are the solutions of the following equations
\begin{align}
&\left\{\begin{array}{l}\displaystyle \frac{\partial z}{\partial t} + Az + f'(y_u)z = v\chi_\omega \text{ in } Q,\\
\partial_{n_A}z = 0 \ \text{ on } \Sigma,\ \ z(0) = 0 \ \text{ in } \Omega,\end{array}\right.\label{E2.7}\\
&\left\{\begin{array}{l}\displaystyle \frac{\partial z}{\partial t} + Az + f'(y_u)z = -f''(y_u)z_{u,v_1}z_{u,v_2} \text{ in } Q,\\
\partial_{n_A}z = 0 \ \text{ on } \Sigma,\ \ z(0) = 0 \ \text{ in } \Omega.\end{array}\right.\label{E2.8}
\end{align}

Let us observe that \eqref{E2.7} has a unique solution $z_{u,v} \in W(0,\infty)$ for every $v \in L^2(Q)$ and there exists a constant $M_3$ such that
\begin{equation}
\|z_{u,v}\|_{L^2(0,\infty;H^1(\Omega))} + \|z_{u,v}\|_{L^\infty(0,\infty;L^2(\Omega))} \le M_3\|v\|_{L^2(Q_\omega)};
\label{E2.9}
\end{equation}see \cite[Theorem A.3]{Casas-Kunisch2023C}.

Using this and the chain rule we deduce that $J:L^2(Q_\omega) \cap L^p(0,\infty;L^2(\omega)) \longrightarrow \mathbb{R}$ is also of class $C^2$ and
\begin{align}
&J'(u)v = \int_{Q_\omega}\varphi_uv \dx\dt,\label{E2.10}\\
&J''(u)(v_1,v_2) = \int_Q[1 - \varphi_uf''(y_u)]z_{u,v_1}z_{u,v_2}\dx\dt,\label{E2.11}
\end{align}
where $\varphi_u \in W(0,\infty) \cap L^\infty(Q)$ is the adjoint state, solution of the equation
\begin{equation}
\left\{\begin{array}{l}\displaystyle -\frac{\partial\varphi}{\partial t} + A^*\varphi + f'(y_u)\varphi = y_u - y_d \text{  in } Q,\\\displaystyle
\partial_{n_{A^*}}\varphi = 0 \ \text{ on } \Sigma,\ \ \lim_{T \to \infty}\|\varphi(T)\|_{L^2(\Omega)} = 0.\end{array}\right.\label{E2.12}
\end{equation}
The reader is referred to \cite[Theorem A.4]{Casas-Kunisch2023C} for the existence, uniqueness, and regularity of $\varphi_u$. In addition, combining \cite[Theorem A.4]{Casas-Kunisch2023C} with \eqref{E2.6} we obtain the existence of a constant $K^j$ such that
\begin{equation}
\|\varphi_u\|_{W(0,\infty)} + \|\varphi_u\|_{L^\infty(Q)} \le K^j \quad \forall u \in \uad^j.
\label{E2.13}
\end{equation}

We note that the linear and bilinear forms $J'(u):L^2(Q_\omega) \cap L^p(0,\infty;L^2(\omega)) \longrightarrow \mathbb{R}$ and $J''(u):[L^2(Q_\omega) \cap L^p(0,\infty;L^2(\omega))]^2 \longrightarrow \mathbb{R}$ can be extended to continuous forms $J'(u):L^2(Q_\omega) \longrightarrow \mathbb{R}$ and $J''(u):L^2(Q_\omega)^2 \longrightarrow \mathbb{R}$ by the same formulas \eqref{E2.10} and \eqref{E2.11}.

\subsection{First-order necessary optimality conditions}
Using Theorem \ref{T2.1}, the proof of the existence of a solution for \Pb follows easily by the classical method of calculus of variations. Since \Pb is not a convex problem, we deal also with local solutions. In particular, we will use the following notion of local solution.

\begin{definition}\label{D2.7}
Given $\bar u \in \uad^j$, $j\in \{1,2\}$, with associated state $\bar y$, we say that $\bar u$ is a strong local solution of \Pb if there exists $\varepsilon > 0$ such that $J(\bar u) \le J(u)$ for every $u  \in \uad^j$ such that $\|y_u - \bar y\|_{L^\infty(Q)} \le \varepsilon$.
\end{definition}
In the sequel, to avoid repetition, we must intend 'strong local solution' whenever we say 'local solution'. As a consequence of \eqref{E2.10}  and the work in \cite[Theorem 2.4]{Casas-Kunisch2023C}, and see also \cite[Theorem 5.1]{zbMATH07563683}, we have that for every local solution $\bar u$, there exist $\bar \varphi\in W(0,\infty)\cap L^\infty(Q)$ such the following optimality system holds
\begin{align}
&\left\{\begin{array}{l}\displaystyle \frac{\partial\bar y}{\partial t} + A\bar y + f(\bar y) = g + \bar u\chi_\omega\ \text{ in } Q,\\
\partial_{n_A}\bar y = 0 \ \text{ on } \Sigma,\ \ \bar y(0) = y_0 \ \text{ in } \Omega,\end{array}\right.\label{E2.14}\\
&\left\{\begin{array}{l}\displaystyle -\frac{\partial\bar\varphi}{\partial t} + A^*\bar\varphi + f'(\bar y)\bar\varphi = \bar y - y_d \text{  in } Q,\\\displaystyle
\partial_{n_{A^*}}\bar\varphi = 0 \ \text{ on } \Sigma,\ \ \lim_{T \to \infty}\|\bar\varphi(T)\|_{L^2(\Omega)} = 0,\end{array}\right.\label{E2.15}\\
&\int_{Q_\omega}\bar\varphi(u - \bar u)\dx\dt \ge 0\quad \forall u \in \uad^j.\label{E2.16}
\end{align}

For the constraint \eqref{E1.2}, we have the following result resembling \cite[Corollary 2.1]{Casas-Kunisch2023C}.
\begin{corollary}
Consider the constraint \eqref{E1.2}. Then, the following properties hold for almost all $t\in (0,\infty)$
\begin{align}
    &\int_\omega \bar \varphi(x,t)(v(x)-\bar u(x,t))\,dx \geq 0 \quad \forall v \in B_{\gamma(t)},\label{E2.17}\\
    &\text{if } \| \bar u(t)\|_{L^2(\omega)} < \gamma(t) \implies \bar \varphi(\cdot, t)=0 \text{ in } \omega,\label{E2.18}\\
    &\text{if } \|\bar \varphi(t)\|_{L^2(\omega)} > 0 \implies \bar u(x,t)=-\gamma(t) \frac{\bar \varphi(x,t)}{\| \bar \varphi(t)\|_{L^2(\omega)}}\label{E2.19}.
\end{align}
\end{corollary}
\begin{proof}
    The inequality \eqref{E2.17} follows from \eqref{E2.16}; see \cite[Theorem 3.3]{Casas-Chrysafinos2026}.
    To prove \eqref{E2.18}, observe that if $\|\bar{u}(t)\|_{L^2(\omega)} < \gamma(t)$, then $\bar{u}(t)$ lies in the interior of $K_{ad}^1(t)$. Choosing test functions $v = \bar{u}(t) \pm \varepsilon \psi$ for arbitrary $\psi \in L^2(\omega)$ and sufficiently small $\varepsilon > 0$ in \eqref{E2.17} yields $\int_\omega \bar{\varphi}(t) \psi \dx = 0$, implying $\bar{\varphi}(t) = 0$.
    To prove \eqref{E2.19}, we observe that the assumption  $\|\bar\varphi(t)\|_{L^2(\omega)} \neq 0$ and \eqref{E2.18} imply that $\|\bar u(t)\|_{L^2(\omega)} = \gamma(t)$. From \eqref{E2.17} we have
 \begin{align*}
 &\|\bar u(t)\|_{L^2(\omega)}\|\bar\varphi(t)\|_{L^2(\omega)} = \gamma(t)\|\bar\varphi(t)\|_{L^2(\omega)} = \sup_{\|v\|_{L^2(\omega)} \le \gamma(t)}-\int_\omega\bar\varphi(t)v\dx\\
 & \le -\int_\omega\bar\varphi(t)\bar u(t)\dx \le \|\bar\varphi(t)\|_{L^2(\omega)}\|\bar u(t)\|_{L^2(\omega)}.
 \end{align*}
 The equality $ -\int_\omega\bar\varphi(t)\bar u(t)\dx = \|\bar\varphi(t)\|_{L^2(\omega)}\|\bar u(t)\|_{L^2(\omega)}$ implies the existence of a constant $c_t < 0$ such that $\bar u(t) = |c_t|\bar\varphi(t)$. Taking norms we get $\gamma(t) = \|\bar u(t)\|_{L^2(\omega)} = c_t\|\bar\varphi(t)\|_{L^2(\omega)}$, what yields $c_t = -\frac{\gamma(t)}{\|\bar\varphi(t)\|_{L^2(\omega)}}$. Then, \eqref{E2.19} follows.
\end{proof}

\section{Second order sufficient optimality conditions}
\label{S3}
\setcounter{equation}{0}
In the rest of this section, $\bar u$ denotes a control of $\uad^j$ satisfying the first order optimality conditions \eqref{E2.14}--\eqref{E2.16} along with the associated state $\bar y$ and adjoint state $\bar\varphi$. We split the section into two parts analyzing the two cases $\uad^j$, $j = 1, 2$, separately.

\subsection{Box-constraints}
For $\tau > 0$, following \cite{Casas-Mateos2020} we define $C^{\tau}_{\bar u} = G^\tau_{\bar u}\cap E^{\tau}_{\bar u }$, where
\begin{align*}
C_{\bar u}&:= \{ v\in L^2(Q_\omega)\vert \  v(x,t) \geq 0 \text{ if } \bar u(x,t)=\alpha,\, v(x,t)\leq 0 \text{ if } \bar u(x,t)=\beta \text{ a.e. in } Q_\omega\},\\
E^\tau_{\bar u}&:= \{v \in L^2(Q_\omega) \vert\ v \in C_{\bar u},\ v(x,t) = 0 \text{ if } |\bar\varphi(x,t)| > \tau \text{ a.e.~in } Q_\omega\},\\
G^\tau_{\bar u}&:=\{v\in L^2(Q_\omega) \vert \ v \in C_{\bar u} \text{ and } J'(\bar u)v\leq \tau \|z_{\bar u,v}\|_{L^1(Q)}\}.
\end{align*}
When $z_{\bar u,v}\notin L^1(Q)$, we have $\|z_{\bar u,v}\|_{L^1(Q)}=\infty$ and the cone $G^\tau_{\bar u}$ does not restrict the direction $v$, nevertheless, for the case $z_{\bar u,v}\in L^1(Q)$, $G^\tau_{\bar u}$ is meaningful. By convention, if $v \in C_{\bar u}$ and $\|z_{\bar u,v}\|_{L^1(Q)}= \infty$, we set $v\in G^\tau _{\bar u}$.

The second-order sufficient optimality condition for \Pb is formulated as follows
\begin{equation}
\exists \tau > 0 \text{ and } \delta > 0 \text{ such that } J''(\bar u)v^2 \ge \delta\|z_{\bar u,v}\|^2_{L^2(Q)}\quad \forall v \in C^\tau_{\bar u},
\label{E3.1}
\end{equation}
where $z_{\bar u,v} = G'(\bar u)v$.

It is well known that a local solution $\bar u$ of \Pb satisfies the second-order necessary condition $J''(\bar u)v^2 \ge 0$ for all $v \in C^0_{\bar u}$; see, for instance, \cite{Casas-Troltzsch2012A}.

\begin{theorem}\label{T3.2}
Under the above assumption \eqref{E3.1}, there exist $\kappa > 0$ and $\varepsilon > 0$ such that
\begin{equation}
J(\bar u) + \frac{\kappa}{2}\|y_u - \bar y\|^2_{L^2(Q)} \le J(u)\quad \forall u \in \uad ^2  \text{ such that } \|y_u - \bar y\|_{L^\infty(Q)} \le \varepsilon.
\label{E3.2}
\end{equation}
\end{theorem}
\begin{proof}
Let $u \in \uad^2$ satisfy $\|y_u - \bar y\|_{L^\infty(Q)} \le \varepsilon$, where $\varepsilon$ will  be determined later independently of $u$.
From Lemma \ref{LA.6} we infer the existence of $\varepsilon_\rho > 0$ such that \eqref{EA.10} holds with $\rho = \frac{\delta}{2}$ for $\|y_u - \bar y\|_{L^\infty(Q)} \le \varepsilon_\rho$. Using Lemma \ref{LA.7} we deduce the existence of $\varepsilon_0 \le \varepsilon_\rho$  such that $\|y_{\bar u + \theta(u - \bar u)} - \bar y\|_{L^\infty(Q)} \le \varepsilon_\rho$ for every $\theta \in (0,1)$ and for all $u \in \uad^2$ with $\|y_u - \bar y\|_{L^\infty(Q)} \le \varepsilon_0$. Hence, we have
\begin{equation}
|[J''(\bar u+\theta(u-\bar u))-J''(\bar u)]v^2| \le \frac{\delta}{2} \quad \forall v \in L^2(Q_\omega).
\label{E3.3}
\end{equation}

Let us consider three different possible cases for $u$.\vspace{2mm}

{\em Case I: $u - \bar u \in  C^{\tau}_{\bar u}$.} Performing a Taylor expansion and using \eqref{E3.1} and \eqref{EA.10} we get for $\|y_u - \bar y\|_{L^\infty(Q)} < \varepsilon_0$ 
\begin{align*}
    J(u)&=
  J(\bar u) +J'(\bar u)(u-\bar u)+\frac{1}{2}J''(\bar u)(u-\bar u)^2\\
  &+\frac{1}{2}[J''(\bar u+\theta(u-\bar u))-J''(\bar u)](u-\bar u)^2\geq J(\bar u) +\frac{\delta}{4}\|z_{\bar u,u-\bar u}\|_{L^2(Q)}^2
\end{align*}

{\em Case II: $u - \bar u\notin G^{\tau}_{\bar u}$.} This implies
\begin{align}
      J(u)&>
  J(\bar u) +\tau \| z_{\bar u,u-\bar u}\|_{L^1(Q)}+\frac{1}{2}J''(\bar u)(u-\bar u)^2\notag\\
  &+\frac{1}{2}[J''(\bar u+\theta(u-\bar u))-J''(\bar u)](u-\bar u)^2.\label{E3.4}
\end{align}
Using \eqref{E3.3} and \eqref{EA.9} we obtain for $\|y_u - \bar y\|_{L^\infty(Q)} \le \varepsilon_0$
\[
    \frac{1}{2}|J''(\bar u)(u-\bar u)^2|+\frac{1}{2}|[J''(\bar u+\theta(u-\bar u))-J''(\bar u)](u-\bar u)^2| \leq \Big(\frac{M_J}{2}+\frac{\delta}{4}\Big)\|z_{\bar u,u-\bar u}\|_{L^2(Q)}^2.
\]
From \eqref{EA.3} we that $\|z_{\bar u,u - \bar u}\|_{L^\infty(Q)} \le (1 + K_{ad,\infty})\|y_u - \bar y\|_{L^\infty(Q)}$. Let us select $\varepsilon_1 = \min\{\varepsilon_0,\frac{2\tau}{(1 + K_{ad,\infty})(M_J + \delta)}\}$. This yields 
\[
\|z_{\bar u, u-\bar u}\|_{L^2(Q)}^2\leq \|z_{\bar u,u-\bar u}\|_{L^\infty(Q)} \|z_{\bar u, u-\bar u}\|_{L^1(Q)} \leq \frac{2\tau}{M_J + \delta}\|z_{\bar u, u-\bar u}\|_{L^1(Q)}
\]
for $\|y_u - \bar y\|_{L^\infty(Q)} \le \varepsilon_1$. Inserting the above two estimates in \eqref{E3.4} we obtain
\begin{align*}
J(u) &\ge  J(\bar u) + \frac{M_J + \delta}{2}\|z_{\bar u,u-\bar u}\|_{L^2(Q)}^2 - \Big(\frac{M_J}{2} + \frac{\delta}{4}\Big)\|z_{\bar u,u-\bar u}\|_{L^2(Q)}^2\\
& = J(\bar u) + \frac{\delta }{4}\|z_{\bar u,u-\bar u}\|_{L^2(Q)}^2.
\end{align*}

{\em Case III: $u - \bar u \in G^\tau_{u - \bar u} \setminus E^\tau_{\bar u}$.} We define
\[
v(x,t) = \left\{\begin{array}{cl}u(x,t) - \bar u(x,t)&\text{if } |\bar\varphi(x,t)| \le \tau,\\0&\text{if } |\bar\varphi(x,t)| > \tau,\end{array}\right.
\]
and $w = (u - \bar u) - v$. It is obvious that $v \in E^\tau_{\bar u}$. 
Moreover, from Lemma \ref{LA.5} we infer the existence of $T_\tau \in (0,\infty)$ independent of $u$ such that $\|\bar\varphi(t)\|_{L^\infty(\Omega)} < \tau$ for every $t > T_\tau$. This implies that $w(x,t) = 0$ for $t > T_\tau$ and $w \in L^1(Q_\omega) \cap L^\infty(0,\infty;L^2(\omega))$. Performing a second order Taylor expansion we get
\begin{align*}
J(u) &= J(\bar u) + J'(\bar u)(u - \bar u) + \frac{1}{2}J''(\bar u + \theta(u - \bar u))(u - \bar u)^2\\
& \ge 
J(\bar u) +J'(\bar u)v+ \tau\|w\|_{L^1(Q)} + \frac{1}{2}J''(\bar u)v^2  + \frac{1}{2}J''(\bar u)w^2 + J''(\bar u)(v,w)\\
&- \frac{1}{2} |[J''(\bar u + \theta(u - \bar u)) - J''(\bar u)](u - \bar u)^2|.
\end{align*}
Let us estimate $\tau\|w\|_{L^1(Q)}$. From the equations satisfied by $\bar\varphi$ and $z_{\bar u,u - \bar u}$ we infer
\small
\begin{align*}
&\tau\|w\|_{L^1(Q_\omega)} \le \int_0^{T_\tau}\int_\omega\bar\varphi(u - \bar u)\dx\dt = \int_\Omega\bar\varphi(T_\tau)z_{\bar u,u - \bar u}(T_\tau)\dx + \int_{Q_{T_\tau}}(\bar y - y_d)z_{\bar u,u - \bar u}\dx\dt\\
&\le \vert\Omega\vert \|\bar\varphi\|_{L^\infty(Q)}\|z_{\bar u,u - \bar u}\|_{L^\infty(Q_{T_\tau})}+\|\bar y - y_d\|_{L^2(Q)}\|z_{\bar u,u - \bar u}\|_{L^2(Q_{T_\tau})}\\
&\le \Big(\vert\Omega\vert \|\bar\varphi\|_{L^\infty(Q)} + \sqrt{|\Omega|T_\tau}\|\bar y - y_d\|_{L^2(Q)}\Big)\|z_{\bar u,u - \bar u}\|_{L^\infty(Q_{T_\tau})}\\
& \le \Big(\vert\Omega\vert \|\bar\varphi\|_{L^\infty(Q)} + \sqrt{|\Omega|T_\tau}\|\bar y - y_d\|_{L^2(Q)}\Big)(1 + K_{ad,\infty})\|y_u - \bar y\|_{L^\infty(Q)} = C_1\|y_u - \bar y\|_{L^\infty(Q)}.
\end{align*}
\normalsize
On the other side, we have with \cite{Casas-Kunisch2025}
\begin{align*}
\|z_{\bar u,w}\|_{L^\infty(Q)} &\le C_2\big(\|w\|_{L^3(Q_\omega))} + \|w\|_{L^2(Q_\omega)}\big)\\
& \le C_2\big(\|\beta - \alpha\|_{L^\infty(Q_\omega)}^{\frac{2}{3}}\|w\|_{L^1(Q_\omega)}^{\frac{1}{3}} + \|\beta - \alpha\|_{L^\infty(Q_\omega)}^{\frac{1}{2}}\|w\|_{L^1(Q_\omega)}^{\frac{1}{2}}\big)\\
& \le C_3(\|w\|_{L^1(Q_\omega)}^{\frac{1}{3}} +\|w\|_{L^1(Q_\omega)}^{\frac{1}{2}}).
\end{align*}
Using this estimate, the above estimate for $\|w\|_{L^1(Q_\omega)}$, and Lemma \eqref{EA.12} we get
\begin{align*}
&\|z_{\bar u,w}\|^2_{L^2(Q)} \le \|z_{\bar u,w}\|_{L^1(Q)} \|z_{\bar u,w}\|_{L^\infty(Q)} \le KC_3(\|w\|_{L^1(Q_\omega)}^{\frac{1}{3}} +\|w\|_{L^1(Q_\omega)}^{\frac{1}{2}})\|w\|_{L^1(Q_\omega)}\\
& \le KC_3\Big(\Big[\frac{C_1}{\tau}\|y_u - \bar y\|_{L^\infty(Q)}\Big]^{\frac{1}{3}} + \Big[\frac{C_1}{\tau}\|y_u - \bar y\|_{L^\infty(Q)}\Big]^{\frac{1}{2}}\Big)\|w\|_{L^1(Q_\omega)}\\
& \le C_4\Big(\|y_u - \bar y\|^{\frac{1}{3}}_{L^\infty(Q)} + \|y_u - \bar y\|^{\frac{1}{2}}_{L^\infty(Q)}\Big)\|w\|_{L^1(Q_\omega)}.
\end{align*}
Now, we need to consider the cases of $v$ being in $G^\tau_{\bar u}$ or not.\vspace{2mm}

{\em Case III-1: $v\in G^\tau_{\bar u}$.} It follows that $v\in C^\tau_{\bar u}$ and we infer from \eqref{E3.1} that $ J''(\bar u)v^2\geq \delta \| z_{\bar u,v}\|_{L^2(Q)}^2$. We take $\varepsilon_2 \le \varepsilon_1$ sufficiently small such that \eqref{EA.10} holds for $\rho = \frac{\delta}{8}$ if $\|y_u - \bar y\|_{L^\infty(Q)} \le \varepsilon_2$, and additionally
\[
 \frac{\tau }{C_4(\sqrt[3]{\varepsilon_2} + \sqrt{\varepsilon_2})} - \frac{M_J}{2}-\frac{\delta}{2}-\frac{2M_J^2}{\delta} \ge 0.
\]
We get with $\| z_{\bar u,v}\|_{L^2(Q)}^2 \ge \frac{1}{2}\| z_{\bar u,u - \bar u}\|_{L^2(Q)}^2 - \| z_{\bar u,w}\|_{L^2(Q)}^2$ and Young's inequality
\begin{align*}
J(u) & \ge 
J(\bar u) + \tau\|w\|_{L^1(Q_\omega)} +\frac{\delta}{2} \| z_{\bar u,v}\|_{L^2(Q)}^2 \\
&-\frac{M_J}{2}\| z_{\bar u,w}\|_{L^2(Q)}^2 - M_J\| z_{\bar u,v}\|_{L^2(Q)}\| z_{\bar u,w}\|_{L^2(Q)}-\frac{\delta}{16}\| z_{\bar u,u-\bar u}\|_{L^2(Q)}^2\\
&\geq J(\bar u) +\left(\frac{\tau }{C_4(\sqrt[3]{\varepsilon_2} + \sqrt{\varepsilon_2})}-\frac{M_J}{2}-\frac{\delta}{2} -\frac{2M_J^2}{\delta}\right)\| z_{\bar u,w}\|_{L^2(Q)}^2+ \frac{\delta}{8}\|z_{\bar u, u - \bar u}\|_{L^2(Q)}^2\\
&- \frac{\delta}{16}\| z_{\bar u,u-\bar u}\|_{L^2(Q)}^2 \ge J(\bar u) + \frac{\delta}{16}\|z_{\bar u,u-\bar u}\|_{L^2(Q)}^2
\end{align*}

{\em Case III-2: $v\notin G^\tau_{\bar u}$.} First, we estimate $\|z_{\bar u,v}\|_{L^\infty(Q)}$. To this end we use the estimates obtained for $z_{\bar u,u - \bar u}$ and $z_{\bar u,w}$ and take $\|y_u - \bar y\|_{L^\infty(Q)} \le \varepsilon$ with $\varepsilon \le \min\{1,\varepsilon_2\}$ to be fixed later:
\begin{align*}
&\|z_{\bar u,v}\|_{L^\infty(Q)} \le \|z_{\bar u,u - \bar u}\|_{L^\infty(Q)} + \|z_{\bar u,w}\|_{L^\infty(Q)}\\
&\le (1 + K_{ad,\infty})\|y_u - \bar y\|_{L^\infty(Q)} + C_3(\|w\|_{L^1(Q_\omega)}^{\frac{1}{3}} +\|w\|_{L^1(Q_\omega)}^{\frac{1}{2}})\\
&\le  (1 + K_{ad,\infty})\varepsilon + C_3\Big(\big[\frac{C_1}{\tau}\varepsilon]^{\frac{1}{3}} + [\frac{C_1}{\tau}\varepsilon\big]^{\frac{1}{2}}\Big)\\
& \le \Big(1 + K_{ad,\infty} + C_3\Big(\big[\frac{C_1}{\tau}\big]^{\frac{1}{3}} + \big[\frac{C_1}{\tau}\big]^{\frac{1}{2}}\Big)\Big)\sqrt[3]{\varepsilon} = C_5\sqrt[3]{\varepsilon}.
\end{align*}
This estimate and the fact that $v\notin G^\tau_{\bar u}$ imply
\[
J'(\bar u)v > \tau\|z_{\bar u,v}\|_{L^1(Q)} \ge \frac{\tau}{C_5\sqrt[3]{\varepsilon}}\|z_{\bar u,v}\|_{L^2(Q)}^2. 
\]
Since $\|y_u - \bar y\|_{L^\infty(Q)} \le \varepsilon$, arguing as in {\em Case III-1} we have that
\[
J'(\bar u)w \ge \tau\|w\|_{L^1(Q_\omega)} \ge \frac{\tau}{2C_4\sqrt[3]{\varepsilon}}\|z_{\bar u,w}\|_{L^2(Q)}^2.
\]
No we additionally impose to $\varepsilon$ to satisfy $\frac{\tau}{2}(\frac{1}{\max\{2C_4,C_5\}\sqrt[3]{\varepsilon}} - M_J) \ge \frac{\delta}{4}$. Then, we have
\begin{align*}
&J(u) - J(\bar u) \ge \tau(\|w\|_{L^1(Q)}+\|z_{\bar u,v}\|_{L^1(Q)}) \\
&-\frac{M_J}{2}\Big(\| z_{\bar u,v}\|_{L^2(Q)}^2 +\|z_{\bar u, w}\|_{L^2(Q)}^2 \Big)-M_J\| z_{\bar u,v}\|_{L^2(Q)}\| z_{\bar u,w}\|_{L^2(Q)}-\frac{\delta}{16} \| z_{\bar u,u-\bar u}\|_{L^2(Q)}^2\\
&\geq \frac{\tau}{2} \Big[\Big(\frac{1}{2C_4\sqrt[3]{\varepsilon}}-M_J\Big)\|z_{\bar u,w}\|_{L^2(Q)}^2+\Big(\frac{1}{C_5\sqrt[3]{\varepsilon}}-M_J\Big)\|z_{\bar u,v}\|_{L^2(Q)}^2\Big] -\frac{\delta}{16} \|z_{\bar u,u - \bar u }\|_{L^2(Q)}^2\\
&\ge \frac{\delta}{4}\Big(\|z_{\bar u,w}\|_{L^2(Q)}^2 + \|z_{\bar u,v}\|_{L^2(Q)}^2\Big) -\frac{\delta}{16} \|z_{\bar u,u - \bar u }\|_{L^2(Q)}^2 \ge \frac{\delta}{16} \|z_{\bar u,u - \bar u }\|_{L^2(Q)}^2.
\end{align*}
\normalsize

We have proved the existence of $\varepsilon >0$ such that $J(u) \geq J(\bar u) + \frac{\delta}{16} \|z_{\bar u,u-\bar u}\|_{L^2(Q)}^2$ for all $u\in \uad^2$ with $\|y_u - \bar y\|_{L^\infty(Q)} \le \varepsilon$. Taking $\varepsilon \le \frac{1}{M_3M_f}$ and applying \eqref{EA.7} we have $\frac{1}{2}\|y_u - \bar y\|_{L^2(Q)} \le \|z_{\bar u,u-\bar u}\|_{L^2(Q)}$. This yields $J(u) \ge J(\bar u) + \frac{\delta}{64}\|y_u - \bar y\|_{L^2(Q)}^2$ and \eqref{E3.2} is fulfilled with $\kappa = \frac{\delta}{32}$.
\end{proof}

\subsection{Pointwise $L^2$-constraints}
Since the control constraints are nonlinear, we are going to define a Lagrangian function associated with problem \Pb. For this purpose we introduce the Hilbert space $L^2_\gamma(Q_\omega)$ of measurable functions $u:Q_\omega \longrightarrow \mathbb{R}$ such that
\[
\|u\|_{L^2_\gamma(Q_\omega)} = \Big(\int_0^\infty\frac{1}{\gamma(t)}\|u(t)\|^2_{L^2(\omega)}\dt\Big)^{\frac{1}{2}} < \infty.
\]
We observe that $L^2_\gamma(Q_\omega) \subset L^2(Q_\omega)$:
\[
\|u\|_{L^2(Q_\omega)}\! \le\! \Big(\!\int_0^\infty\frac{\|\gamma\|_{L^\infty(0,\infty)}}{\gamma(t)}\|u(t)\|^2_{L^2(\omega)}\dt\!\Big)^{\frac{1}{2}}\! \le\! \|\gamma\|^{\frac{1}{2}}_{L^\infty(0,\infty)}\|u\|_{L^2_\gamma(Q_\omega)}\, \forall u \in L^2_\gamma(Q_\omega).
\]

We also notice that $\uad^1 \subset L^2_\gamma(Q_\omega)$:
\[
\int_0^\infty\frac{1}{\gamma(t)}\|u(t)\|^2_{L^2(\omega)}\dt \le \int_0^\infty\gamma(t)\dt \quad \forall u \in \uad^1.
\]
We define the Lagrangian function $\mathcal{L}:L^p(0,\infty;L^2(\omega)) \cap L^2_\gamma(Q_\omega) \times L^\infty(0,\infty) \longrightarrow \mathbb{R}$
\[
\mathcal{L}(u,\lambda) = J(u) + \frac{1}{2}\int_0^\infty\frac{\lambda(t)}{\gamma(t)}\|u(t)\|^2_{L^2(\omega)}\dt.
\]
$\mathcal{L}$ is of class $C^2$ and using \eqref{E2.10} and \eqref{E2.11} we get
\begin{align}
&\frac{\partial\mathcal{L}}{\partial u}(u,\lambda)v = \int_{Q_\omega}\varphi_uv\dx\dt + \int_0^\infty\frac{\lambda(t)}{\gamma(t)}\int_\omega uv\dx\dt,\label{E3.5}\\
&\frac{\partial^2\mathcal{L}}{\partial u^2}(u,\lambda)(v_1,v_2)\notag\\& = \int_Q(1 - \varphi_uf''(y_u))z_{u,v_1}z_{u,v_2}\dx\dt + \int_0^\infty\frac{\lambda(t)}{\gamma(t)}\int_\omega v_1v_2\dx\dt.\label{E3.6}
\end{align}
We note that the expression \eqref{E3.5} allows the extension of the linear  form $\frac{\partial\mathcal{L}}{\partial u}(u,\lambda)$ $L^2_\gamma(Q_\omega)$ in the case where $u \in \uad^1$ and $\lambda \in L^2(0,\infty)$. It is enough to observe that
\begin{align*}
&\int_0^\infty\frac{\lambda(t)}{\gamma(t)}\int_\omega u(t)v(t)\dx\dt \le \int_0^\infty|\lambda(t)|\|v(t)\|_{L^2(\omega)}\frac{\|u(t)\|_{L^2(\omega)}}{\gamma(t)}\dt\\
& \le \int_0^\infty|\lambda(t)\|v(t)\|_{L^2(\omega)}\dt \le \|\lambda\|_{L^2(0,\infty)}\|v\|_{L^2(Q_\omega)}.
\end{align*}

We define $\bar\lambda:(0,\infty)\to[0,\infty)$ by $\bar\lambda(t) = \|\bar\varphi(t)\|_{L^2(\omega)}$. We observe that $\bar\lambda \in L^2(0,\infty) \cap C[0,\infty)$ and from \eqref{E2.18}  we have
\[
\bar\lambda(t) \ge 0\ \ \text{ and }\ \ \bar\lambda(t)[\|\bar u(t) - \gamma(t)] = \quad \forall t \in (0,\infty).
\]
Let us denote
\[
I_\gamma = \{t \in (0,\infty) : \|\bar u(t)\|_{L^2(\omega)} = \gamma(t)\} \text{ and } I^+_\gamma = \{t \in I_\gamma : \bar\lambda(t) > 0\}.
\]
These properties along with the next proposition proposition show that $\bar\lambda$ is a Lagrange multiplier associated with $\bar u$.
\begin{proposition} The following property holds:
\begin{equation}
\frac{\partial\mathcal{L}}{\partial u}(\bar u,\bar\lambda)v = 0\quad \forall v \in L^2(Q_\omega).
\label{E3.7}
\end{equation}
\label{P3.1}
\end{proposition}

\begin{proof}
From \eqref{E3.5}, the definition of $\bar\lambda$, and \eqref{E2.19}  we infer
\begin{align*}
\frac{\partial\mathcal{L}}{\partial u}(\bar u,\bar\lambda)v &= \int_0^\infty\int_\omega\bar\varphi(t)v(t)\dx\dt + \int_0^\infty\frac{\bar\lambda(t)}{\gamma(t)}\int_\omega\bar u(t)v(t)\dx\dt\\
&= \int_{I_\gamma^+}\int_\omega\Big(\bar\varphi(t) + \frac{\|\bar\varphi(t)\|_{L^2(\omega)}}{\gamma(t)}\bar u(t)\Big)v(t)\dx\dt = 0.
\end{align*}
\end{proof}

Now, we introduce the cone of critical directions associated with $\bar u$: for $\tau > 0$
\[
C^\tau_{\bar u} :=\left\{v\in L^2_\gamma(Q_\omega) : \left\{\begin{array}{l}\displaystyle \int_\omega\bar u(t) v(t)\dx \le 0 \text{ if } t \in I_\gamma,\vspace{2mm}\\\displaystyle \int_0^\infty\frac{\bar\lambda(t)}{\gamma(t)}\int_\omega\bar u(t)v(t)\dx\dt \ge -\tau\|z_{\bar u,v}\|_{L^1(Q)}.\end{array}\right.\right\}.
\]
We assume that
\begin{equation}
\exists \delta > 0 \text{ and } \exists\tau > 0 \text{ such that } \frac{\partial^2\mathcal{L}}{\partial u^2}(\bar u,\bar\lambda)v^2 \ \ge\  \delta\| z_{\bar u,v}\|^2_{L^2(Q)}\quad \forall v\in C^\tau_{\bar u}.
\label{E3.8}
\end{equation}

\begin{theorem}
\label{T3.2}
Under assumption \eqref{E3.8} there exist $\varepsilon>0$ and $\kappa > 0$ such that 
\begin{equation}\label{E3.9}
J(u)\ \ge\ J(\bar u)\ +\ \frac{\kappa}{2}\|y_u-\bar y\|_{L^2(Q)}^2\quad \forall u\in \uad^1 \text{ such that } \|y_u-\bar y\|_{L^\infty(Q)}\le \varepsilon.
\end{equation}
\end{theorem}
\begin{proof}
Let $u \in \uad^1$. We assume that $\|y_u - \bar y\|_{L^\infty(Q)} \le \varepsilon$, where $\varepsilon > 0$ is selected, as in the proof of Theorem \ref{T3.2}, in such a way that \eqref{EA.10} holds with $\rho = \frac{\delta}{2}$. Now, we split the proof into two cases.\vspace{2mm}

{\em Case I: $u - \bar u \in C^\tau_{\bar u}$.} We observe that $\|u(t)\|_{L^2(\omega)} \le \gamma(t) = \|\bar u(t)\|_{L^2(\omega)}$ if $\bar\lambda(t) > 0$. Then, we have with \eqref{E3.6}, \eqref{E3.7}, and \eqref{E3.8}
\begin{align*}
&J(u) - J(\bar u) \ge \mathcal{L}(u,\bar\lambda) - \mathcal{L}(\bar u,\bar\lambda) = \frac{1}{2}\frac{\partial^2\mathcal{L}}{\partial u^2}(\bar u + \theta(u - \bar u),\bar\lambda)(u - \bar u)^2\\
&= \frac{1}{2}\frac{\partial^2\mathcal{L}}{\partial u^2}(\bar u,\bar\lambda)(u - \bar u)^2 + \frac{1}{2}[J''(\bar u + \theta(u - \bar u),\bar\lambda) - J''(\bar u)](u - \bar u)^2\\
& \ge \frac{\delta}{2}\|z_{\bar u,u - \bar u}\|_{L^2(Q_\omega)}^2 - \frac{\delta}{4}\|z_{\bar u,u - \bar u}\|_{L^2(Q_\omega)}^2 = \frac{\delta}{4}\|z_{\bar u,u - \bar u}\|_{L^2(Q_\omega)}^2.
\end{align*}

{\em Case II: $u - \bar u \notin C^\tau_{\bar u}$.} We have that for $t \in I^+_\gamma$
\begin{align*}
&\int_\omega\bar u(t)(u(t) - \bar u(t))\dx \le \|\bar u(t)\|_{L^2(\omega)}\|u(t)\|_{L^2(\omega)} - \|\bar u(t)\|^2_{L^2(\omega)}\\
& = \gamma(t)[\|u(t)\|_{L^2(\omega)} - \gamma(t)] \le 0.
\end{align*}
Hence, since $u - \bar u \notin C^\tau_{\bar u}$ we have that
\[
0 = \frac{\partial\mathcal{L}}{\partial u}(\bar u,\bar\lambda)(u - \bar u) < J'(\bar u)(u - \bar u) - \tau\|z_{\bar u,u - \bar u}\|_{L^1(Q_\omega)}.
\]
Therefore, we get that $J'(\bar u)(u - \bar u) > \tau\|z_{\bar u,u - \bar u}\|_{L^1(Q_\omega)}$. Performing a Taylor expansion we obtain with \eqref{EA.9}
\begin{align*}
&J(u) - J(\bar u) = J'(\bar u)(u - \bar u) + \frac{1}{2}J''(u_\theta)(u - \bar u)^2\\
& \ge \tau\|z_{\bar u,u - \bar u}\|_{L^1(Q)} - \frac{M_J}{2}\|z_{u_\theta,u - \bar u}\|^2_{L^2(Q)}.
\end{align*}
where $u_\theta = \bar u + \theta(u- \bar u)$.
From \eqref{EA.4} we infer $\|z_{u_\theta,u - \bar u}\|_{L^2(Q)} \le C_{ad}\varepsilon\|z_{\bar u, u - \bar u}\|_{L^2(Q)}$ with $C_{ad} = 1 + 2M_3M_fM_{ad}$. Once again, using \eqref{EA.6} we get
\begin{align*}
&\|z_{\bar u,u - \bar u}\|^2_{L^2(Q)} \le \|z_{\bar u,u - \bar u}\|_{L^\infty(Q)}\|z_{\bar u,u - \bar u}\|_{L^1(Q)}\\
& \le (1 + K_{ad,\infty})\|y_u - \bar y\|_{L^\infty(Q)}\|z_{\bar u,u - \bar u}\|_{L^1(Q)} \le  (1 + K_{ad,\infty})\varepsilon\|z_{\bar u,u - \bar u}\|_{L^1(Q)}.
\end{align*}
Using the last inequalities we deduce
\[
J(u) - J(\bar u) \ge \Big[\frac{\tau}{ (1 + K_{ad,\infty})\varepsilon} - \frac{M_J}{2}C_{ad}^2\varepsilon^2\Big]\|z_{\bar u,u - \bar u}\|^2_{L^2(Q)}.
\]
Taking $\varepsilon$ small enough we deduce that $J(u) \ge J(\bar u)+ \frac{\delta}{4}\|z_{\bar u,u - \bar u}\|^2_{L^2(Q)}$. Finally, \eqref{E3.9} is obtained with $\kappa = \frac{\delta}{8}$ as at the end of the proof of Theorem \ref{T3.2}.
\end{proof}

\section{Approximation by finite horizon problems}
\label{S4}
In this section, we consider the approximation of solutions of the infinite-horizon problem by  solutions to finite horizon problems. For an infinite horizon problem where a Tikhonov term is present, this was done in \cite{Casas-Kunisch2023C}. 
Let ($\bar u,\bar y)$ denote a local solution of $(\text{P})$. For every $0 < T < \infty$, $j\in \{1,2\}$ and $\rho>0$, we define the finite-horizon control problem
\[
\PbT \quad \min_{u \in \uadt^j} J_T(u): =\frac{1}{2} \int_{Q_T} \left( y_{T,u} - y_d \right)^2 \dx\dt,
\]
where 
\[
\uadt^j := \{u\in L^2(Q_{T,\omega})\vert \ u(\cdot, t) \in K_{\mathrm{ad}}^j(t)\text{ a.a. } t\in (0,T)\},\ \ Q_{T,\omega} = \omega \times (0,T),
\]
and $y_{T,u}$ denotes the solution of the equation
\begin{equation}
\left\{\begin{array}{l}\displaystyle \frac{\partial y}{\partial t} +A y + f( y) 
= g + u \, \chi_{\omega} \ \ \text{in } Q_T, \\
 \partial_{n_A}y = 0 \ \text{on } \Sigma_T = \Gamma \times (0, T),\ \ 
y(0) = y_0  \text{ in } \Omega.
\end{array}\right.
\label{E4.1}
\end{equation}

We extend any $u \in \uadt^j$ to $\uad^j$ by setting $u(t) = 0$ for $t > T$, denoted as $\hat{u}$. Next we prove that problems \PbT realize an approximation of \Pb.

\begin{theorem}
For every $T < \infty$, \PbT has at least one solution. If $\{u_T\}_{T > 0}$ is a family of solutions of problems \PbT, then there exist sequences $\{u_{T_k}\}_{k = 1}^\infty$ with $T_k \to \infty$ as $k \to \infty$ such that $\hat u_{T_k} \stackrel{*}{\rightharpoonup} \bar u$ in $L^\infty(0,\infty;L^2(\omega)) \cap L^2(Q_\omega)$. Moreover, every of these limit points $\bar u$ is a solution of \Pb and in addition the convergence $y_{\hat u_{T_k}} \to \bar y= y_{\bar u}$ strongly in $L^\infty(Q)$ holds.
\label{T4.1}
\end{theorem}
The proof of the existence of a solution for every problem \PbT is standard. The boundedness of the admissible sets $\uad^j$ in $L^\infty(0,\infty;L^2(\omega)) \cap L^2(Q_\omega)$ implies tha existence of sequences $\hat u_{T_k} \stackrel{*}{\rightharpoonup} \bar u$. From Theorem \ref{T2.1} we infer the convergence  $y_{\hat u_{T_k}} \to \bar y$ strongly in $L^\infty(Q)$. Finally the proof that $\bar u$ is a solution of \Pb is immediate; see \cite[Theorem 4.1]{Casas-Kunisch2023C} for a similar proof under the presence of the Tikhonov term.

\begin{remark}
Let $\{u_{T_k}\}_{k = 1}^\infty$ be a sequence of functions with $u_{T_k} \in L^2(Q_{T_k,\omega})$ and $T_k \to \infty$. Let $v \in L^2(Q_\omega)$ be arbitrarily chosen. We denote by $\hat u_{T_k}$ and $\tilde u_{T_k}$ the extensions of $u_{T_k}$ to $Q_\omega$ by 0 and by $v$, respectively. Then, we have that $\hat u_{T_k} \rightharpoonup u$ in $L^2(Q_\omega)$ if and only if $\tilde u_{T_k} \rightharpoonup u$ in $L^2(Q_\omega)$. Indeed, it is enough to notice that for every $w \in L^2(Q_\omega)$
\[
\Big|\int_{Q_\omega}w(\tilde u_{T_k} - \hat u_{T_k})\dx\dt\Big| = \Big|\int_{Q_{T_k,\omega}}wv\dx\dt\Big| \le \|w\|_{L^2(Q_{T_k,\omega})}\|v\|_{L^2(Q_{T_k,\omega})} \stackrel{k \to \infty}{\longrightarrow} 0.
\]
\label{R4.1}
\end{remark}

We have a converse theorem.

\begin{theorem}
Let $\bar u$ be a strict strong local solution of \Pb. Then, there exists a family $\{u_T\}_{T > 0}$ of local solutions of problems \PbT such that $\hat u_T \stackrel{*}{\rightharpoonup} \bar u$ in $L^\infty(0,\infty;L^2(\omega)) \cap L^2(Q_\omega)$ as $T \to \infty$ and $J_T(u_T) \le J_T(\bar u)$. Moreover, the convergence $y_{\hat u_{T_k}} \to \bar y= y_{\bar u}$ strongly in $L^\infty(Q)$ holds.
\label{T4.2}
\end{theorem}

\begin{proof}
Since $\bar u$ is a strict strong local minimizer of \Pb, there exists $\rho > 0$ such that
\[
J(\bar u) < J(u)\quad \forall u \in \uad^j \text{ such that } \|y_u - \bar y\|_{L^\infty(Q)} \le \rho.
\]
We define the control problems
\[
\Pbr \quad \min_{u \in \uad^j,\ \|y_u - \bar y\|_{L^\infty(Q)} \le \rho} J(u) \quad\text{and}\quad \PbTr \quad \min_{u \in \uadt^j,\ \|y_{T,u} - \bar y\|_{L^\infty(Q_T)} \le \rho} J_T(u).
\]
It is obvious that $\bar u$ is the unique solution of \Pbr. On the other hand, the restriction of $\bar u$ to $Q_{T,\omega}$ is an admissible control for problem \PbTr. Therefore, the set of admissible points of \PbTr is nonempty and closed with respect to the weak$^*$ topology of $L^\infty(0,\infty;L^2(\omega)) \cap L^2(Q_\omega)$. Then, the proof the existence of a solution $u_T$ follows with Theorem \ref{T2.1}. Let $\{u_{T_k}\}_{k = 1}^\infty$ be a sequence of solutions of problems \PbTr with $T_k \to \infty$ as $k \to \infty$. Since $\hat u_{T_k} \in \uad^j$ for every $k$, we deduce that, taking a subsequence, $\hat u_{T_k} \stackrel{*}{\rightharpoonup} \tilde u$ in $L^\infty(0,\infty;L^2(\omega)) \cap L^2(Q_\omega)$. Using again Theorem \ref{T2.1} we have that $y_{u_{T_k}} \to y_{\tilde u}$ strongly in $L^\infty(Q)$. Since $\|y_{u_{T_ k}} - \bar y\|_{L^\infty(Q)} \le \rho$ for every $T_k > T$ we infer that  $\|y_{\tilde u} - \bar y\|_{L^\infty(Q)} \le \rho$ for all $T < \infty$. Therefore, we have that $\|y_{\tilde u} - \bar y\|_{L^\infty(Q)} \le \rho$ and consequently $\tilde u$ is a feasible point of \Pbr. Passing to the limit in the inequality $J_{T_k}(u_{T_k}) \le J_{T_k}(\bar u)$ we deduce that $J(\tilde u) \le J(\bar u)$. Since $\bar u$ is the unique solution of \Pbr, we infer that $\tilde u = \bar u$ and the theorem follows.
\end{proof}

\begin{theorem}
Let $\bar u$ a local solution of problem \Pb satisfying the second-order sufficient optimality condition. Let $\{u_T\}_{T > 0}$ be a sequence of local minimizers of problems \PbT as introduced in Theorem \ref{T4.2}. Then, there exist $T^* < \infty$ and a constant $C$ independent of $T$ such that for every $T \ge T^*$
\begin{equation}
\|y_{u_T} - \bar y\|_{L^2(Q_T)} \le C\Big(\|y_{u_T}(T)\|_{L^2(\Omega)} + \|g\|_{L^2(\Omega \times (T,\infty))} + \|y_d\|_{L^2(\Omega \times (T,\infty))}\Big).
\label{E4.2}
\end{equation}
\label{T4.3}
\end{theorem}

\begin{proof}
Let us denote by $\hat u_T$ the extension of $u_T$ by zero to $Q_\omega$. Since $\hat u_T \rightharpoonup \bar u$ in $L^2(Q_\omega)$ and $\{\hat u_T\}_{T > 0}$ is bounded in $L^\infty(0,\infty;L^2(\omega))$, then we infer that $\hat u_T \rightharpoonup \bar u$ in $L^p(0,\infty;L^2(\omega)$ as well. From Theorem \ref{T2.1} the convergence $y_{\hat u_T} \to \bar y$ in $L^\infty(Q)$ follows. Let $\varepsilon > 0$ be given by \eqref{E3.2} or \eqref{E3.9} and take $T^* < \infty$ such that $\|y_{\hat u_T} - \bar y\|_{L^\infty(Q)} \le \varepsilon$ for all $T \ge T^*$. Then, we have
\[
J(\bar u) + \frac{\kappa}{2}\|y_{\hat u_T} - \bar y\|^2_{L^2(Q)} \le J(\hat u_T)\quad \forall T \ge T^*.
\]
Then, using the optimality of $u_T$ we get from the above inequality
\begin{align}
&\|y_{T,u_T} - \bar y\|_{L^2(Q_T)} \le \Big(\frac{2}{\kappa}\Big[J_T(u_T) - J_T(\bar u) + \frac{1}{2}\int_T^\infty\|y_{\hat u_T}(t) - y_d(t)\|^2_{L^2(\Omega)}\Big]\Big)^{\frac{1}{2}}\notag\\
&\le \frac{1}{\sqrt{\kappa}}\Big(\|y_{\hat u_T}\|_{L^2(\Omega \times (T,\infty))} + \|y_d\|_{L^2(\Omega \times (T,\infty))}\Big).
\label{E4.3}
\end{align}
Let us estimate $\|y_{\hat u_T}\|_{L^2(\Omega \times (T,\infty))}$. We notice that $y_{\hat u_T}$ satisfies
\[
\left\{\begin{array}{l}\displaystyle \frac{\partial y_{\hat u_T}}{\partial t} +A y_{\hat u_T} + f(y_{\hat u_T}) 
= g \ \ \text{in } \Omega \times T,\infty), \\
\partial_{n_A}y_{\hat u_T} = 0 \ \text{on } \Gamma \times (T,\infty),\ \ 
y_{\hat u_T}(T) = y_{T,u_T}(T)  \text{ in } \Omega.
\end{array}\right.
\]
This implies the existence of a constant $C_0$ such that
\[
\|y_{\hat u_T}\|_{L^2(\Omega \times (T,\infty))} \le C_0\Big(\|y_{u_T}(T)\|_{L^2(\Omega)} + \|g\|_{L^2(\Omega \times (T,\infty))}\Big).
\]
Inserting this in \eqref{E4.3} we infer \eqref{E4.2} for $C = \max\{C_0,1\}/\sqrt{\kappa}$.
\end{proof}

\appendix
\section{Estimates for the state and adjoint equations}
In order to prove the main theorems of this paper, we establish in the following, technical results. Along this appendix $\bar u$ is a fixed control of $\uad^j$ for $j = 1$ or 2, and $\bar y$ and $\bar\varphi$ are its associated state and adjoint state.

\begin{lemma}
The following properties hold for all $u \in \uad^j$, $j\in\{1,2\}$, and all $v \in L^2(Q_\omega)$
\begin{align}
&\|z_{u,v} - z_{\bar u,v}\|_{L^2(Q)} \le M_3M_f\|y_u - \bar y\|_{L^\infty(Q)}\|z_{\bar u,v}\|_{L^2(Q)},\label{EA.1}\\
&\|y_u - \bar y - z_{\bar u,u - \bar u}\|_{L^2(Q)}\le\frac{M_3M_f}{2}\|y_u - \bar y\|_{L^\infty(Q)}\|y_u - \bar y\|_{L^2(Q)},\label{EA.2}\\
&\|y_u - \bar y - z_{\bar u,u - \bar u}\|_{L^\infty(Q)} \le K_{ad,\infty}\|y_u - \bar y\|_{L^\infty(Q)}.\label{EA.3}
\end{align}
where $M_3$ was introduced in \eqref{E2.9} and $M_f = \sup_{|s| \le M_{ad}}|f''(s)|$, $M_{ad}$ being defined in \eqref{E2.6}.
\label{LA.1}
\end{lemma}

\begin{proof}
Let us prove \eqref{EA.1}. We set $w = z_{u,v} - z_{\bar u,v}$. Subtracting the equations satisfied by $z_{u,v}$ and $z_{\bar u,v}$ we get
\[
\left\{\begin{array}{l}\displaystyle \frac{\partial w}{\partial t} + Aw + f'(y_u)w = [f'(\bar y) - f'(y_u)]z_{\bar u,v} \text{ in } Q,\\
\partial_{n_A}w = 0 \ \text{ on } \Sigma,\ \ w(0) = 0 \ \text{ in } \Omega.\end{array}\right.
\]
From \eqref{E2.9} and \eqref{E2.6} we deduce with the mean value theorem that
\[
\|w\|_{L^2(Q)} \le M_3\|f'(\bar y) - f'(y_u)\|_{L^\infty(Q)}\|z_{\bar u,v}\|_{L^2(Q)} \le M_3M_f\|y_u - \bar y\|_{L^\infty(Q)}\|z_{\bar u,v}\|_{L^2(Q)}.
\]

Now, taking $w = y_u - \bar y - z_{\bar u,u - \bar u}$ we get
\[
\left\{\begin{array}{l}\displaystyle \frac{\partial w}{\partial t} + Aw + f'(\bar y)w = -\frac{1}{2}f''(\hat y)(y_u - \bar y)^2 \text{ in } Q,\\
\partial_{n_A}w = 0 \ \text{ on } \Sigma,\ \ w(0) = 0 \ \text{ in } \Omega,\end{array}\right.
\]
where $\hat y = \bar y + \theta(y_u - \bar y)$ for some measurable function $\theta: Q \longrightarrow [0,1]$. Using again \eqref{E2.9} we obtain 
\[
\|w\|_{L^2(Q)} \le \frac{M_3}{2}\|f''(\hat y)(y_u - \bar y)^2\|_{L^2(Q)} \le \frac{M_3}{2}M_f\|y_u - \bar y\|_{L^\infty(Q)}\|y_u - \bar y\|_{L^2(Q)}.
\]
which proves \eqref{EA.2}. To establish \eqref{EA.3} we apply \cite{Casas-Kunisch2025} and use \eqref{E2.6}  to the above equation to get
\begin{align*}
&\|w\|_{L^\infty(Q)} \le \frac{C}{2}\Big(\|f''(\hat y)(y_u - \bar y)^2\|_{L^2(Q)} + \|f''(\hat y)(y_u - \bar y)^2\|_{L^\infty(Q)}\Big)\\
&\le \frac{C}{2}\Big(M_f\|y_u - \bar y\|_{L^2(Q)}\|y_u - \bar y\|_{L^\infty(Q)} + M_f\|y_u - \bar y\|^2_{L^\infty(Q)}\Big)\\
& \le [CM_{ad}M_f(M_3 + 1)]\|y_u - \bar y\|_{L^\infty(Q)},
\end{align*}
which proves \eqref{EA.3} with $K_{ad,\infty} = CM_{ad}M_f(M_3 + 1)$.
\end{proof}

\begin{lemma}
The following properties hold for all $u \in \uad^j$, $j\in\{1,2\}$, and all $v \in L^2(Q_\omega)$
\begin{align}
&\|z_{u,v}\|_{L^2(Q)} \le (1 + 2M_3M_fM_{ad})\|z_{\bar u,v}\|_{L^2(Q)},\label{EA.4}\\
&\|z_{\bar u,u - \bar u}\|_{L^2(Q)} \le (1 + M_3M_fM_{ad})\|y_u - \bar y\|_{L^2(Q)},\label{EA.5}\\
&\|z_{\bar u,u - \bar u}\|_{L^\infty(Q)} \le (1 + K_{ad,\infty})\|y_u - \bar y\|_{L^\infty(Q)},\label{EA.6}\\
&\|y_u - \bar y\|_{L^2(Q)} \le 2\|z_{\bar u,u - \bar u}\|_{L^2(Q)}\quad \text{if}\quad \|y_u - \bar y\|_{L^\infty(Q)} \le \frac{1}{M_3M_f}.\label{EA.7}
\end{align}
\label{LA.21}
\end{lemma}
This lemma is an immediate consequence of Lemma \ref{LA.1}. In fact, \eqref{EA.4} follows from \eqref{EA.1}, \eqref{EA.5} and \eqref{EA.7} are deduced from \eqref{EA.2}, and \eqref{EA.6} follows from \eqref{EA.3}.

\begin{lemma}
For every $\rho > 0$ there exists $\varepsilon_\rho > 0$ such that
\begin{equation}
\|y_u - \bar y\|_{L^2(Q)} \le \rho\quad \forall u \in \uad^j, j \in \{1, 2\},  \text{ such that } \|y_u - \bar y\|_{L^\infty(Q)} \le \varepsilon_\rho.
\label{EA.8}
\end{equation}
\label{LA.3}
\end{lemma}

\begin{proof}
We set $w = y_u - \bar y$. Then, we have
\[
\left\{\begin{array}{l}\displaystyle \frac{\partial w}{\partial t} + Aw + f'(\bar y + \theta(y_u - \bar y))w = (u - \bar u)\chi_\omega \text{ in } Q,\\
\partial_{n_A}w = 0 \ \text{ on } \Sigma,\ \ w(0) = 0 \ \text{ in } \Omega.\end{array}\right.
\]
As in the proof of Theorem \ref{T2.1}, we take the function $h \in L^2(0,\infty) \cap L^\infty(0,\infty)$ such that $\|v(t)\|_{L^2(\omega)} \le h(t)$ for almost all $t \in (0,\infty)$ and all $v \in \uad^j$. Then, for every $T < \infty$ we deduce with \eqref{E2.9} that
\begin{align*}
\|w\|_{L^2(Q^T)} &\le M_3\Big(\|u - \bar u\|_{L^2(\omega \times (T,\infty))} + \|w(T)\|_{L^2(\Omega)}\Big)\\
&\le M_3\Big(2\|h\|_{L^2(\omega \times (T,\infty))} + \sqrt{|\Omega|}\|y_u - \bar y\|_{L^\infty(Q)}\Big).
\end{align*}
We select $T_\rho < \infty$ such that $M_32\|h\|_{L^2(\omega \times (T,\infty))} < \frac{\rho}{2}$. We also take $\varepsilon_\rho > 0$ such that $(\sqrt{T_\rho} + M_3)\sqrt{|\Omega|}\varepsilon < \frac{\rho}{2}$. Then, we get
\[
\|y_u - \bar y\|_{L^2(Q)} \le \|y_u - \bar y\|_{L^2(Q_{T_\rho})} + \|y_u - \bar y\|_{L^2(Q^{T_\rho})} \le \rho
\]
for all $u \in \uad^j$, $j \in \{1, 2\}$,  such that $\|y_u - \bar y\|_{L^\infty(Q)} \le \varepsilon_\rho$.
\end{proof}

\begin{lemma}\label{LA.4}
For every $\rho  > 0$ there exists $\varepsilon_\rho > 0$ such that $\|\varphi_u -\bar\varphi\|_{L^\infty(Q)} \le \rho$ for all $u \in \uad^i$, $i\in \{1,2\}$, satisfying $\|y_u - \bar y\|_{L^\infty(Q)} \le \varepsilon_\rho$.
\end{lemma}

\begin{proof}
Let $u \in \uad^i$ be fixed and set $\phi := \varphi_u - \bar\varphi$. For the function $\phi$, we obtain
\[
\left\{\begin{array}{l}
\displaystyle - \frac{\partial\phi}{\partial t} + A\phi + f'(\bar y)\phi =  y_u - \bar y + [f'(\bar y) - f'(y_u)]\,\varphi_u\ \text{ in } Q,\\
\partial_{n_A}\phi = 0 \ \text{ on } \Sigma,\ \lim_{T\to\infty}\|\phi(T)\|_{L^2(\Omega)} = 0.
\end{array}\right.
\]
By the mean value theorem, \eqref{E2.6}, and recalling the definition of $M_f$ in Lemma \ref{LA.1} we get
\[
\|f'(\bar y) - f'(y_u)\|_{L^\infty(Q)} \le M_f\|y_u - \bar y\|_{L^\infty(Q)} .
\]
Then, from \cite[Theorem A.4]{Casas-Kunisch2023C} and \eqref{E2.13} we obtain
\begin{align*}
&\|\phi\|_{L^\infty(Q)} \le C\Big(\|y_u - \bar y\|_{L^2(Q)} + \|y_u - \bar y\|_{ L^p(0,\infty;L^2(\Omega))}\\
& + \|[f'(\bar y) - f'(y_u)]\varphi_u\|_{L^2(Q)} +  \|[f'(\bar y) - f'(y_u)]\varphi_u\|_{ L^p(0,\infty;L^2(\Omega))}\Big)\\
&\le C\Big(\|y_u - \bar y\|_{L^2(Q)} + \|y_u - \bar y\|_{L^\infty(0,\infty;L^2(\Omega))}^{\frac{p - 2}{p}}\|y_u - \bar y\|^{\frac{2}{p}}_{L^2(Q)}\\
&+ M_f\|y_u - \bar y\|_{L^\infty(Q)}\|\varphi_u\|_{L^2(Q)} + M_f\|y_u - \bar y\|_{L^\infty(Q)}\|\varphi_u\|_{L^p(0,\infty;L^2(\Omega))}\Big)\\
&\le C\Big(\|y_u - \bar y\|_{L^2(Q)} + [\sqrt{|\Omega|}\|y_u - \bar y\|_{L^\infty(Q)}]^{\frac{p - 2}{p}}[2M_{ad}]^{\frac{2}{p}}\\
&\hspace{3.1cm} + 2(1 + |\Omega|)^{\frac{p -2}{2p}}M_fK^j\|y_u - \bar y\|_{L^\infty(Q)}\Big)\\
&\le C\Big(\|y_u - \bar y\|_{L^2(Q)} + [\sqrt{|\Omega|}\varepsilon]^{\frac{p - 2}{p}}[2M_{ad}]^{\frac{2}{p}} + 2(1 + |\Omega|)^{\frac{p -2}{2p}}M_fK^j\varepsilon\Big).
\end{align*}
From Lemma \ref{LA.3} we infer the existence of $\varepsilon_1 > 0$ such that $\|y_u - \bar y\|_{L^2(Q)} \le \frac{\rho}{2C}$ if $\|y_u - \bar y\|_{L^\infty(Q)} \le \varepsilon_1$. We also select $\varepsilon_2 > 0$ such that
\[
C\Big([\sqrt{|\Omega|}\varepsilon_2]^{\frac{p - 2}{p}}[2M_{ad}]^{\frac{2}{p}} + 2(1 + |\Omega|)^{\frac{p -2}{2p}}M_fK^j\varepsilon_2\Big) \le \frac{\rho}{2}.
\]
Then, it is enough to take $\varepsilon_\rho = \min\{\varepsilon_1,\varepsilon_2\}$ to get that $\|\varphi_u -\bar\varphi\|_{L^\infty(Q)} \le \rho$.
\end{proof}

\begin{lemma}\label{LA.5}
For every $\rho > 0$ there exists $T_\rho \in (0,\infty)$ such that $\|\bar\varphi\|_{L^\infty(Q^{T_\rho})} \le \rho$.
\end{lemma}

\begin{proof}
Recall that the adjoint state satisfies $\lim_{t\to\infty}\|\bar\varphi(t)\|_{L^2(\Omega)} = 0$. Fix $T>1$ and define $\eta_T:(0,\infty)\to\mathbb R$ by
\[
\eta_T(t) :=
\begin{cases}
0, & 0<t<T-1,\\
t+1-T, & T-1 \le t \le T,\\
1, & t>T.
\end{cases}
\]
Set $\phi_T(t,x) := \eta_T(t)\,\bar\varphi(t,x)$.
Then $\phi_T$ satisfies $
\phi_T(t,x) = 0$ for $t<T-1$, $\phi_T(t,x) = \bar\varphi(t,x)$ for $ t>T$,
and $\|\bar\varphi\|_{L^\infty(Q^T)} \le \|\phi_T\|_{L^\infty(Q)}$. Further we have
\[
\left\{
\begin{array}{l}
- \partial_t\phi_T + A\phi_T + f'(\bar y)\phi_T = h_T \ \text{ in } Q,\\
\partial_{n_A}\phi_T = 0 \text{ on } \Sigma,\ \lim_{t\to\infty}\|\phi_T(t)\|_{L^2(\Omega)} = 0,
\end{array}\right.
\]
where $h_T(t,x) = \eta_T(t)(\bar y - y_d)(t,x) - \eta_T'(t)\,\bar\varphi(t,x)$.
Here $\eta_T'$ is equal to $1$ on $(T-1,T)$ and $0$ elsewhere. From \cite[Theorem A.4]{Casas-Kunisch2023C} we deduce the existence of a constant $C$ independent of $T$ such that
\begin{align*}
\|\phi_T\|_{L^\infty(Q)} &\le C\Big( \|h_T\|_{L^2(Q)}  + \|h_T\|_{ L^p(0,\infty;L^2(\Omega))} \Big)\\
&\le C\Big(\|\bar y - y_d\|_{L^2(Q^{T - 1})} + \|\bar y - y_d\|_{L^p(T - 1,\infty;L^2(\Omega))}\\
&+ \|\bar\varphi\|_{L^2(\Omega \times (T-1,T)} + \|\bar\varphi\|_{L^p(T-1,T;L^2(\Omega))}\Big) \stackrel{T \to \infty}{\longrightarrow} 0,
\end{align*}
which proves the statement.
\end{proof}

\begin{lemma}\label{LA.6}
The second derivative $J''$ enjoys the following properties:

1.- For $M_J = 1 +K^jM_f$ we have
\begin{equation}
|J''(u)(v_1,v_2)| \le M_J\|z_{u,v_1}\|_{L^2(Q)}\|z_{u,v_2}\|_{L^2(Q)}\ \forall v_1,v_2 \in L^2(Q_\omega).
\label{EA.9}
\end{equation}

 2.- For every $\rho > 0$ there exists $\varepsilon_\rho > 0$ such that for all $u \in \uad^j$, $j\in\{1,2\}$, with $\|y_u - \bar y\|_{L^\infty(Q)} < \varepsilon_\rho$ we have that
\begin{equation}
|[J''(u) - J''(\bar u)]v^2| \le \rho\|z_{\bar u,v}\|^2_{L^2(Q)}\quad \forall v \in L^2(Q_\omega).
\label{EA.10}
\end{equation}
\end{lemma}
\begin{proof}
The inequality \eqref{EA.9} is a straightforward consequence of \eqref{E2.11}, \eqref{E2.15}, and the definition of $M_f$. Let us prove \eqref{EA.10}. From \eqref{E2.11} we get
\begin{align*}
&[J''(u) - J''(\bar u)]v^2 = \int_Q(z^2_{u,v} - z^2_{\bar u,v})\dx\dt + \int_Q(\bar\varphi - \varphi_u)f''(y_u)z^2_{u,v}\dx\dt\\
& + \int_Q\bar\varphi(f''(\bar y) - f''(y_u))z^2_{u,v}\dx\dt + \int_Q\bar\varphi f''(\bar y)(z^2_{u,v} - z^2_{\bar u,v})\dx\dt\\
& = I_1 + I_2 + I_3 + I_4.
\end{align*}
To estimate $I_1$ we select $0 < \varepsilon_1 \le [2M_3M_f(1 + M_3M_fM_{ad})]^{-1}\frac{\rho}{4}$ and for $\|y_u - \bar y\|_{L^\infty(Q)} < \varepsilon_1$ we  get with \eqref{EA.1} and \eqref{EA.4}
\[
|I_1| \le \|z_{u,v} - z_{\bar u,v}\|_{L^2(Q)}\|z_{u,v} + z_{\bar u,v}\|_{L^2(Q)} \le \frac{\rho}{4}\|z_{\bar u,v}\|^2_{L^2(Q)}.
\]
Let us estimate $I_2$. From Lemma \ref{LA.4} we get the existence of $\varepsilon_2 > 0$ such that
\[
\|\bar\varphi - \varphi_u\|_{L^\infty(Q)} \le \frac{\rho}{4(1 + 2M_3M_fM_{ad})^2M_f)}\quad \forall u \in \uad^j \text{ with } \|y_u - \bar y\|_{L^\infty(Q)} < \varepsilon_2.
\]
Then, using again \eqref{EA.4} we have
\[
|I_2| \le \|\bar\varphi - \varphi_u\|_{L^\infty(Q)}M_f\|z_{u,v}\|^2_{L^2(Q)}  \le \frac{\rho}{4}\|z_{\bar u,v}\|^2_{L^2(Q)}.
\]
For $I_3$ we select $\varepsilon_3 > 0$ such that for all $s_1, s_2 \in [-M_{ad},+M_{ad}]$ with $|s_2 - s_1| \le \varepsilon_3$
\[
|f''(s_2) - f''(s_1)| \le \frac{\rho}{4\|\bar\varphi\|_{L^\infty(Q)}(1 + 2M_3M_fM_{ad})^2}.
\]
Then, we get $|I_3| \le \|\bar\varphi\|_{L^\infty(Q)}\|f''(y_u) - f''(\bar y)\|_{L^\infty(Q)}\|z_{u,v}\|^2_{L^2(Q)} \le \frac{\rho}{4}\|z_{\bar u,v}\|^2_{L^2(Q)}$ for $\|y_u - \bar y\|_{L^\infty(Q)} \le \varepsilon_3$.

Finally, the estimate of $I_4$ follows in the same way as for the estimate of $I_1$ because $|I_4| \le \|\bar\varphi\|_{L^\infty(Q)}\|f''(\bar y)\|_{L^\infty(Q)}|I_1|$. It is enough to take $0 < \varepsilon_4 \le \frac{\varepsilon_1}{\|\bar\varphi\|_{L^\infty(Q)}\|f''(\bar y)\|_{L^\infty(Q)}}$.  Hence, the proof is concluded by setting $\varepsilon_\rho = \min \varepsilon_i$.
\end{proof}

\begin{lemma}\label{LA.7}
For every $\rho > 0$ there exists $\varepsilon_\rho > 0$ such that for all $u \in \uad^j$, $j\in \{1,2\}$, satisfying $\|y_u - \bar y\|_{L^\infty(Q)} \le \varepsilon_\rho$, the following inequality holds
\begin{equation}
\|y_{\bar u + \theta(u - \bar u)} - \bar y\|_{L^\infty(Q)} \le \rho \quad \forall \theta \in (0,1).
\label{EA.11}
\end{equation}
\end{lemma}

The reader is referred to \cite[Lemma 3.5]{CMR2019} for the proof.

\begin{lemma}\label{LA.8}
There exists a constant $K$ such that
\begin{equation}\label{EA.12}
\|z_{\bar u,v}\|_{L^1(Q)} \le K\,\|v\|_{L^1(Q_\omega)}
\qquad \forall\, v\in L^2(Q_\omega)\cap L^1(Q_\omega).
\end{equation}
\end{lemma}

\begin{proof}
For every $T < \infty$ we define the functions $\phi_T \in W(0,T) \cap L^\infty(Q_T)$ as the solution of the equation
\[
\left\{
\begin{array}{l}
- \partial_t \phi_T + A\phi_T + f'(\bar y)\,\phi_T = 1 \text{ in } Q_T,\\
\partial_{n_{A}}\phi_T = 0 \text{ on } \Sigma,\ \ \phi_T(T)  = 0 \text{ in } \Omega,
\end{array}
\right.
\]
From the maximum principle we infer that $\phi_T \ge 0$. We also define the function $\psi \in H^1(\Omega) \cap L^\infty(\Omega)$ solution of
\[
\left\{
\begin{array}{l} A\psi = 1 \text{ in }\Omega,\\
\partial_{n_A}\psi = 0 \text{ on } \Gamma.\end{array}\right.
\]
We obviously have that $\psi \ge 0$. Finally, we set $\psi_T = \phi_T - \psi$ and get
\[
\left\{
\begin{array}{l}
- \partial_t \psi_T + A\psi_T + f'(\bar y)\,\psi_T = -f'(\bar y)\psi \text{ in } Q_T,\\
\partial_{n_{A}}\psi_T = 0 \text{ on } \Sigma,\ \ \psi_T(T)  = -\psi \text{ in } \Omega.
\end{array}
\right.
\]
Since $-f'(\bar y)\psi \le 0$ and $-\psi \le 0$ we deduce that $\psi_T \le 0$ and, consequently, $0 \le \phi_T \le \psi$. This implies that $\|\phi_T\|_{L^\infty(Q_T)}  \le \|\psi\|_{L^\infty(\Omega)}$ for every $T < \infty$.

Now, let us assume that $v \in L^2(Q_\omega) \cap L^1(Q_\omega)$ and $v \ge 0$. Then, we have that $z_{\bar u,v} \ge 0$. From the equations satisfied by $z_{\bar u,v}$ and $\phi_T$, and integrating by parts we get
\[
\int_{Q_T}z_{\bar u,v} \dx\dt = \int_{Q_T}v\phi_T\dx\dt \le \|v\|_{L^1(Q_\omega)}\|\phi_T\|_{L^\infty(Q_T)} \le \|\psi\|_{L^\infty(\Omega)}\|v\|_{L^1(Q_\omega)}.
\]
Thus, we have that $\|z_{\bar u,v}\|_{L^1(Q_T)} \le \|\psi\|_{L^\infty(\Omega)}\|v\|_{L^1(Q_\omega)}$. For a general function $v \in L^1(Q_\omega) \cap L^2(Q_\omega)$, we take $v = v^+ - v^-$ and get
\[
\|z_{\bar u,v}\|_{L^1(Q_T)} \le \|\psi\|_{L^\infty(\Omega)}(\|v^+\|_{L^1(Q_\omega)} + \|v^-\|_{L^1(Q_\omega)}) = \|\psi\|_{L^\infty(\Omega)}\|v\|_{L^1(Q_\omega)}.
\]
Taking the supremum as $T \to \infty$ we deduce \eqref{EA.12} with $K = \|\psi\|_{L^\infty(\Omega)}$.
\end{proof}

\bibliographystyle{siamplain}
\bibliography{Eduardo-Nicolai}
\end{document}